\documentclass[11pt]{amsart}
\usepackage{amsmath,amssymb,latexsym,esint,cite,mathrsfs}
\usepackage{verbatim,wasysym}
\usepackage[left=2.8cm,right=2.8cm,top=2.85cm,bottom=2.85cm]{geometry}

\usepackage{microtype}
\usepackage{color,enumitem,graphicx}
\usepackage[colorlinks=true,urlcolor=blue, citecolor=red,linkcolor=blue,
linktocpage,pdfpagelabels, bookmarksnumbered,bookmarksopen]{hyperref}
\usepackage[hyperpageref]{backref}
\usepackage[english]{babel}

\newcommand{\eqlab}[1]{\begin{equation}  \begin{aligned}#1 \end{aligned}\end{equation}} 
		\newcommand{\bgs}[1]{\begin{equation*} \begin{aligned}#1\end{aligned}\end{equation*}} 
				\newcommand{\syslab}[2] []  {\begin{equation}#1  \left\{\begin{aligned}#2\end{aligned}\right.\end{equation}} 
						\newcommand{\sys}[2][]{\begin{equation*}#1  \left\{\begin{aligned}#2\end{aligned}\right.\end{equation*}}

\newcommand{\R}{\mathbb{R}}

\renewcommand{\epsilon}{{\varepsilon}}

\newtheorem{lemma}{Lemma}[section]

\newtheorem{proposition}[lemma]{Proposition}
\newtheorem{theorem}[lemma]{Theorem}

\theoremstyle{definition}

\newtheorem{remark}[lemma]{Remark}
\newtheorem{definition}[lemma]{Definition}

\newtheorem{problem}{Problem}

\newcommand{\dive}{{\rm div}}
\numberwithin{equation}{section}

\usepackage{orcidlink} 
		\usepackage[square, comma, numbers, sort&compress]{natbib} 

\title[Concavity principles for nonautonomous equations]{Concavity principles for nonautonomous elliptic equations and applications}

\author[N.\ Almousa]{Nouf Almousa \orcidlink{0000-0001-6412-1993}}
\author[C.\ Bucur]{Claudia Bucur \orcidlink{0000-0002-5061-7338}}
\author[R.\ Cornale]{Roberta Cornale}
\author[M.\ Squassina]{Marco Squassina \orcidlink{0000-0003-0858-4648}}

\address[C.\ Bucur]{\newline\indent Dipartimento di Matematica Federigo Enriques
	\newline\indent
	Università degli Studi di Milano
	\newline\indent
	Italy, Milano, Via Saldini 50, 20133}
\email{\href{mailto:claudia.bucur@unimi.it}{claudia.bucur@unimi.it}} 

\address[N.\ M.\ Almousa, M.\ Squassina]{\newline\indent College of Science
	\newline\indent
	Princess Nourah Bint Abdul Rahman University
	\newline\indent
	Saudi Arabia, Riyadh, PO Box 84428}
\email{\href{mailto:nmalmousa@pnu.edu.sa}{nmalmousa@pnu.edu.sa}}
\email{\href{mailto:marsquassina@pnu.edu.sa}{marsquassina@pnu.edu.sa}}

\address[R.\ Cornale, M.\ Squassina]{\newline\indent Dipartimento di Matematica e Fisica
	\newline\indent
	Università Cattolica del Sacro Cuore
	\newline\indent
	Italy, Brescia, Via della Garzetta 48, 25133}
\email{\href{mailto:roberta.cornale01@unicatt.it}{roberta.cornale01@unicatt.it}} 
\email{\href{mailto:marco.squassina@unicatt.it}{marco.squassina@unicatt.it}}

\thanks{The second and forth authors are members
	of {\em Gruppo Nazionale per l'Analisi Ma\-te\-ma\-ti\-ca, la Probabilit\`a e le loro Applicazioni} (GNAMPA) 
	of the {\em Istituto Nazionale di Alta Matematica} (INdAM). Authors Almousa and Squassina are supported by Princess Nourah bint Abdulrahman University Researchers Supporting
	Project number (PNURSP-HC2023/3), Princess Nourah bint Abdulrahman University, Saudi Arabia}

\begin{document}

\begin{abstract}
In the study of concavity properties of positive solutions to nonlinear elliptic partial differential equations
the diffusion and the nonlinearity are typically independent of the space variable. In this paper we obtain 
new results aiming to get almost concavity results for a relevant class of anisotropic semilinear elliptic 
problems with spatially dependent source and diffusion.
\end{abstract}

\maketitle

\section{introduction}

A rather natural question in the field of nonlinear partial differential equations is whether a positive solution with homogeneous Dirichlet boundary conditions
is concave on a given convex domain.  
Starting from \cite{GNN79} 
extensive research has been developed in order to deduce symmetry of solutions 
from the symmetry of the domain, via the so called Alexandroff-Serrin moving plane method.
When the symmetry of the domain is dropped, one may wonder if the solutions still inherit some concavity properties from the domain.
This was investigated in a series of pioneering papers \cite{PS51, 
	ML71, BL76, Lew77}.

When studying concavity properties of solutions, it becomes evident that aiming for concavity is often overly demanding: while it can be achieved for the torsion problem, for example, for suitable perturbation of ellipsoids \cite{HNST18,Ken85,Kor83Co,Lind94},
the first eigenfunctions of the Laplacian are never concave, regardless of the considered bounded domain  \cite[Remark 3.4]{Kaw85R}.
One may instead search for a strictly increasing function $\varphi$ that, when composed with the solution $u$, yields a concave function $\varphi(u)$.
In the seminal paper \cite{ML71} 
it is shown that the solutions of the torsion problem $-\Delta u=1$ are such that $\sqrt{u}$ is concave. In \cite{BL76} 
the authors show that the positive eigenfunctions of $-\Delta u=\lambda u$ satisfy that $\log u$ is concave. The concavity of solutions to nonlinear equations has been  explored in several subsequent papers \cite{Lio81,Kor83Co, Kaw85W, Ken85,Sak87,Lin94,Gom07}
involving techniques 
mainly relying on maximum principles applied to suitably defined convexity functions.
For instance if $\beta\in (0,1)$, $\Omega$ is  convex and $u$ is a positive solution to $-\Delta u=u^\beta$
with Dirichlet boundary data,
then $u^{(1-\beta)/2}$ is concave \cite{Ken85}.

Most of the cited papers, however, give assumptions on the nonlinearity in the equation in order to have a suitable power $u^{\gamma}$ of the solution $u$ to be concave. Recently, 
for the problem $-\Delta u=f(u)$ (and more generally for quasi-linear problems
involving the $p$-Laplace operator),  under suitable assumptions on $f$ the authors of \cite{BMS22} showed concavity of
$\int_1^u {1/\sqrt{F(\sigma)}} d\sigma$, where $F'=f$, thus providing a precise connection on how the concavity of the solution is affected by the nonlinear term $f$. In \cite{AAGS}, these results were then extended to  
the quasi-linear problem 
$	-\dive(\alpha(u) \nabla u) + \frac{1}{2} \alpha'(u) |\nabla u|^2 = f(u) $
related to the so called modified nonlinear Schrödinger equation, under suitable joint
hypothesis on $\alpha$ and $f$. More precisely $\int_{\mu}^u \sqrt{\alpha(\sigma)/F(\sigma)} d\sigma$ turns out to be concave for some positive constant $\mu$.

In cases where the assumptions on the function $f$ which guarantee the concavity of a suitable transformation are not met, some
quantitative perturbation results were recently obtained in \cite{BS20}. These results establish, in essence, a bound on the loss of concavity of $u$, controlled  in the supremum norm in terms of the loss of concavity of $f$.

In general all the results in the current literature only deal with {\em autonomous} problems, corresponding to isotropic physical models, namely
both the diffusion term in the operator and the nonlinearity {\em do not} explicitly depend upon the space variable
and it is expected that concavity (even up to a transformation) is in general broken due to the $x$-dependence.

The primary objective of the paper is to establish quantitative perturbation results, which assert that if both the diffusion term in the operator and the nonlinearity exhibit a small variation with respect to the spatial variable, then  a suitable transformation $\varphi(u)$ is  close to a concave function in the supremum norm, with an error estimate depending precisely on  the spatial variation. 

Precisely, taking $\Omega\subset \R^n$ a bounded open strictly convex set with smooth boundary, consider the semi-linear problem, for $ \beta\in [0,1)$,
	  \syslab[]{\label{popf1}
	 -&\sum_{i,j=1}^n \alpha^{ij}(x)D^2_{ij}u=a(x) \, u^\beta \quad && \mbox{ in } \Omega\\
	& u>0 &&\mbox{ in } \Omega\\
	& u=0 && \mbox{ on } \partial \Omega,
}
with the matrix of coefficients  $A=\{\alpha^{ij}\}_{i,j=1}^n$ being symmetric and uniformly elliptic. In the isotropic cases $\alpha^{ij}=\delta_{ij}$ and $a=1$ this reduces to the
already mentioned classical sublinear problem $-\Delta u=u^\beta$ for which  a result in \cite{Ken85} establishes concavity of $u^{(1-\beta)/2}$ .

As a by product of a general maximum convexity principle (see Theorem \ref{lemma2}) we prove in Proposition \ref{mainappl} that if
\eqlab{ \label{pppr}
\|\nabla a\|_{L^{\infty}(\Omega)}+\max_{i,j\in \{1, \dots, n\}}\|\nabla \alpha^{ij}\|_{L^{\infty}(\Omega)}< \epsilon,\qquad \epsilon >0,
}
then there exists a positive constant $C$ and a concave function $w:\Omega\to\R$ such that
$$
\left\|u^{\frac{1-\beta}{2}}-w\right\|_{L^{\infty}(\Omega)}\leq C\epsilon,
$$
in light also of a Hyers-Ulam theorem (see Proposition	\ref{ulam}).

Furthermore, as a second example, consider the problem 
\syslab[]{\label{popf2}
	& - \sum_{i,j=1}^n \alpha^{ij}(x)D^2_{ij}u=a(x)u+\varepsilon \varphi(u) && \mbox{ in } \Omega\\
	& u>0 &&\mbox{ in } \Omega\\
	& u=0 && \mbox{ on } \partial \Omega,
}
with   $A=\{\alpha^{ij}\}_{i,j=1}^n$ symmetric and uniformly elliptic. If \eqref{pppr} holds and $\varphi$ is non-increasing, then there exists a positive constant $C$ and a concave function $w\colon \Omega \to \R$ such that
\[ \|\log u -w\|_{L^\infty(\Omega)} \leq C \epsilon.\]  We obtain this applying Proposition \ref{ghfd}. 
We point out that the case $\varphi=0$ in problem \eqref{popf2}, which corresponds to $\beta=1$ in problem \eqref{popf1}, is out of reach 
since our general convexity maximum principles fail, precisely since assumption \eqref{cres} is not fulfilled.

In the rest of the paper, we proceed obtaining some maximum principles for concavity functions of solutions of semi-linear equations, which can be viewed as anisotropic counterparts of the results presented in \cite[Lemma 1.4]{Kor83Co} and \cite[Lemma 3.1]{Ken85}. We then discuss some applications, precisely problems \eqref{popf1} and \eqref{popf2}.  We believe that
our techniques could be suitable to investigate other physically relevant anisotropic elliptic problems. To the best of our knowledge this is the first result in the literature providing almost concavity results for anisotropic problems in convex domains. 

\medskip

\section{Anisotropic convexity principles}

In the rest of the paper, let $\Omega$ denote a bounded open  convex subset of $\R^n$. Denote furthermore for $x_{1},x_{3} \in \overline \Omega$, $\lambda \in [0,1]$,
\begin{equation} \label{x2}  x_{2}:=\lambda x_{3}+(1-\lambda)x_{1} \in \overline \Omega
\end{equation}
and for $s_1,s_3\in \R$
\begin{equation*} 
 s_2= \lambda s_3+ (1-\lambda) s_1.
 \end{equation*}
For some $u\colon \overline \Omega \to \R$, we define the concavity function $\mathcal C_u$ as
\begin{equation} \label{varfi}
	\mathcal C_u (x_{1},x_{3},\lambda):=u(x_{2})-\lambda u(x_{3})-(1-\lambda)u(x_{1}).
\end{equation}
For some $g\colon \overline \Omega \times \R\to \R$, the joint-concavity function $\mathcal {JC}_g$ is defined by
\begin{equation}\label{joint}
	\mathcal {JC}_g((x_1,s_1), (x_3,s_3),\lambda) := g(x_2,s_2)  - \lambda g(x_3,s_3)-(1-\lambda) g(x_1,s_1),
\end{equation} 
and we will also use the notation
\begin{equation}\label{joint}
	\mathcal {JC}_{g(\cdot, u(\cdot))}(x_1,x_3,\lambda) := g(x_2,\lambda u(x_3)+(1-\lambda) u(x_1))  - \lambda g(x_3,u(x_3))-(1-\lambda) g(x_1,u(x_1))
\end{equation}
when $s_i=u(x_i)$.
We define the harmonic  concavity function, as in \cite{Ken85}, in the following way:
\syslab[ \mathcal{HC}_{g} ( (y_1,s_1),(y_3,s_3),\lambda):=]{ \label{harmonic} 
		& g(
		y_2,s_2)
		- \frac{g(y_1,s_1) g(y_3,s_3)}{\lambda g(y_1,s_1) + (1-\lambda )g(y_3,s_3)},
		\\
		& \qquad\qquad\qquad\qquad \mbox{ if }   \lambda g(y_1,s_1) + (1-\lambda) g(y_3,s_3)>0
		\\
		 & g(
		 y_2,s_2),\qquad\quad\;\;\, \,\mbox{ if } 
		  g(y_1,s_1) =  g(y_3,s_3)=0.
		  }	
It should be noted that  such definition is applicable to positive functions $g$, or functions that can change sign and that meet one of the conditions specified in equation \eqref{harmonic}, at the point $((y_1, s_1), (y_3, s_3), \lambda)$.  Notice also that if $g<0$, none of these conditions are satisfied.

We will also use the notation
\begin{equation*}
	\mathcal {HC}_{g(\cdot, u(\cdot))}(x_1,x_3,\lambda)= g(x_2,\lambda u(x_3)+(1-\lambda) u(x_1)) - \frac{g(x_1,u(x_1))g(x_3,u(x_3))}{\lambda g(x_1,u(x_1))+ (1-\lambda)g(x_3,u(x_3))}
\end{equation*}
when $s_i=u(x_i)$.
Notice that $\mathcal C_u, \mathcal {JC}_g , \mathcal {HC}_g \geq 0$ are equivalent to the concavity, joint concavity, respectively harmonic concavity of the functions. 

To ensure clarity, we also point out the following definition. 

\begin{definition}\label{x2}
We say that the triple  $(x_1,x_3,\lambda)$  is an interior point for $\mathcal C_u$ if  each of $x_1,x_2,x_3 $ is in $ \Omega$ with $x_2=\lambda x_3+(1-\lambda) x_1$, while we say that the point is on the boundary if at least one $x_1,x_2,x_3 $ belongs to $\partial  \Omega $. 
\end{definition} 

 Having established our notations, 
we point out how we obtain our almost-concavity results for transformations of the  solutions  of \eqref{popf1}, \eqref{popf2}. It is obvious that  if $u\in C(\overline \Omega)$, then $\mathcal C_u$ achieves a maximum in $\overline \Omega \times \overline \Omega \times [0,1]$. We give in this section maximum convexity principles, which cover the case in which $\mathcal C_u$  achieves a positive maximum at an interior point in $\Omega \times \Omega \times (0,1)$. To follow, in Section \ref{sec2}, after noticing that that the concavity functions associated to our problems, due to boundary constraints, cannot achieve the positive maximum on the boundary, with a direct applications of the maximum convexity principles we obtain the desired conclusion.

We introduce now the model problem for which we obtain maximum convexity principles. 

\begin{problem}\label{problem}
For all $i,j\in \{1, \dots,n\}$ let the functions
\[ a^{ij} \colon \overline \Omega \times \R^n \to \R\] be such that
	$a^{ij} (\cdot, \xi) \in C^1( \Omega)$ and $A=[a^{ij}(x,\xi)]_{i,j=1}^n$ is a symmetric positive semidefinite matrix for all $(x,\xi) \in \overline \Omega \times \R^n$.
	Let $b\colon \overline \Omega \times \R \times \R^n \to \R$ be such that $b(x,\cdot, \xi) $ is differentiable in $\R\setminus\{0\}$, for all $(x, \xi) \in \overline \Omega \times \R^n$. Consider
the equation
\eqlab{ \label{eqn} &Lu=0,\qquad Lu=a^{ij}(x,Du)u_{ij}-b(x,u,Du),
} 
where we use the notation
\[ a^{ij} u_{ij}:= \sum_{i,j=1}^n a^{ij}u_{ij}.\] 
\end{problem}

The next result, an anisotropic maximum convexity principle, can be viewed as the  anisotropic counterpart of \cite[Lemma 2.3]{BS20}, both variations of the classical convexity principle in \cite[Lemma 1.4]{Kor83Co}. 


\begin{theorem}
\label{lemma1}
	Let $\Omega\subset \R^n$ be a bounded open  convex set. Let $u  \in  C^{2} ( \Omega)$ 
	 be a solution of Problem \ref{problem}.  Assume that  $\mathcal C_u$ achieves a positive interior maximum at $(x_1,x_3,\lambda) \in \Omega \times \Omega \times (0,1) $.
	If  there is some $\sigma >0$ such that for all $x$ on the segment $[x_1,x_3]$ and $s$ on the segment $[u(x_1),u(x_3)]$ it holds that
	\eqlab{\label{cres}  \frac{\partial b}{\partial s}(x,s,Du(x_1)) \geq \sigma,}
then
	\[\mathcal C_u(x_1,x_3,\lambda) \leq  -\frac{\mathcal {JC}_{b(\cdot, u(\cdot),\xi)}(x_1,x_3,\lambda) }{\sigma} +\frac{C \varepsilon(D(u(x_1))}{\sigma},\]
	where
	\eqlab{\label{epsi}\varepsilon(Du(x_1)) := \max_{i,j\in\{1, \dots,n\} } \sup_{x\in [x_1,x_3]} |D_xa^{ij}(x,Du(x_1))| }
	and
	\eqlab{\label{const} 
		C:= n^2 \max_{i,j \in\{1,\dots,n\} } \max_{ k \in \{ 1,3\} }|u_{ij}(x_k) | \, {\rm diam}(\Omega) > 0.}
\end{theorem}

\begin{proof}
Notice that if $ x_{1} = x_{3} $, then the inequality trivially holds. We may hence assume that $ x_{1}, x_{2}, x_{3} $ are distinct.  Since $\mathcal C_u$ achieves a maximum at $(x_1,x_3,\lambda)$, recalling \eqref{x2} and \eqref{varfi}, we get that
	\[ (D_{x_{1}}\mathcal C_u)(x_{1},x_{3},\lambda) = (D_{x_{3}}\mathcal C_u)(x_{1},x_{3},\lambda) =0 \]
	hence
	\[ (1-\lambda)Du(x_{2})-(1-\lambda)Du(x_{1}) =
	\lambda Du(x_{2})- \lambda Du(x_{3}) =0 .\] 
	Let us set
	\[\xi :=  Du(x_{1})=Du(x_{2})=Du(x_{3}),\]
	and consider  the auxiliary  function $ \bar{\varphi} \colon \R^n \to \R$ defined as 
	\[  \bar{\varphi}(v):=\mathcal C_u (x_{1}+v,x_{3}+v,\lambda)=u(x_{2}+v)-\lambda u(x_{3}+v)-(1-\lambda) u(x_{1}+v) .\] 
	Since $\bar{\varphi} $ has a local maximum at $ v=0 $, we get that
	\[ \nabla_{v}\bar{\varphi}(0)=0 \qquad \mbox{ and } \quad  [D_{v}^{2}\bar{\varphi}(0)]\leq0.\]
	We recall that
	if $A$ and $B$ are two $n\times n$ real symmetric positive semidefinite matrices, then ${\rm Tr}(AB)\geq 0$ (see \cite[Lemma A.1]{Ken85}). 
	Since $A=[a^{ij}(x_2,\xi)]_{i,j=1}^n$ is positive semidefinite, it follows that
	\begin{equation*}
		a^{ij}(x_{2},\xi)(u_{ij}(x_{2})-\lambda  u_{ij}(x_{3})-(1-\lambda)u_{ij}(x_{1}))\leq0. 
	\end{equation*}  
	Denote
\begin{equation}\label{err13}
		\begin{aligned}
			&e_{1}=(a^{ij}(x_{2},\xi)-a^{ij}(x_{1},\xi))u_{ij}(x_{1}),
			\\
			&e_{3}=(a^{ij}(x_{2},\xi)-a^{ij}(x_{3},\xi))u_{ij}(x_{3}).
		\end{aligned}
	\end{equation}
	and using the equation \eqref{eqn}, we have 
	\begin{align*} 
		b(x_{2},u(x_{2}), \xi)&=a^{ij}(x_{2},\xi)u_{ij}(x_{2})\leq  
		\lambda a^{ij}(x_{2},\xi)u_{ij}(x_{3})+(1-\lambda)a^{ij}(x_{2},\xi)u_{ij}(x_{1}) \\ 
		&= \lambda a^{ij}(x_{3},\xi)u_{ij}(x_{3})+\lambda e_{3}+(1-\lambda)a^{ij}(x_{1},\xi)u_{ij}(x_{1})+(1-\lambda) e_{1} \\
		&= \lambda b(x_{3},u(x_{3}),\xi)+(1-\lambda)b(x_{1},u(x_{1}),\xi)+(1-\lambda)e_{1}+\lambda e_{3}.
	\end{align*} 
	So we get in turn 
	\begin{equation*}
		b(x_{2},u(x_{2}),\xi)-b(x_{2},\lambda u(x_{3})+(1-\lambda )u(x_1),\xi)\leq  
	\end{equation*}
	\begin{equation*}\label{b}
		\lambda b(x_{3},u(x_{3}),\xi)+(1-\lambda)b(x_{1},u(x_{1}),\xi)-b(x_{2},\lambda u(x_{3})+(1-\lambda )u(x_1),\xi)+(1-\lambda)e_{1}+\lambda e_{3}.
	\end{equation*}
	Using the Lagrange's theorem, we can estimate 
	\begin{equation}\label{stima99}
		\max \{ |e_1|, |e_3|\} \}\leq C \varepsilon(\xi),
	\end{equation}
	so we get that
	\begin{equation*}
		(1-\lambda)e_{1}+\lambda e_{3}\leq(1-\lambda)|e_{1}|+\lambda|e_{3}|=C\varepsilon(\xi).
	\end{equation*}
	Then we can apply Lagrange's theorem to obtain that there exists  $\bar s$ on the segment $[u(x_{2}),\lambda u(x_{3})+(1-\lambda)u(x_{1})]$, thus on the segment $ [u(x_1), u(x_3)]$, such that
	\[
	\sigma 	\mathcal C_u(x_1,x_3,\lambda) \leq \frac{\partial b}{\partial s} (x_2, \overline s, \xi)\left(u(x_{2})-\lambda u(x_{3})-(1-\lambda)u(x_{1}) \right)\leq -\mathcal{JC}_{b(\cdot, u(\cdot), \xi)}(x_1,x_3,\lambda) +C\varepsilon(\xi),\] 
		 concluding the proof of the Theorem.
\end{proof}


\bigskip

We have now the second anisotropic approximate convexity principle,  counterpart of \cite[Lemma 2.9]{BS20}, both variations of the classical Convexity Principle in \cite[Lemma 3.1]{Ken85}. 
%

\begin{theorem}
	\label{lemma2}
	Let $\Omega\subset \R^n$ be a bounded open  convex set. Let $u  \in  C^{2} (  \Omega)$
	 be a solution of Problem \ref{problem}.  Assume that  $\mathcal C_u$ achieves a positive interior maximum at $(x_1,x_3,\lambda) \in \Omega \times \Omega \times (0,1)$, and additionally that there is some  $\nu, \sigma>0$ such that for all $x$ on the segment $[x_1,x_3]$ and $s$ on the segment $[u(x_1),u(x_3)]$ it holds that
	\eqlab{\label{onb} 
		b(x,s,Du(x_1))\geq \nu }
	and
	\eqlab{ \label{onderivb} \frac{\partial b}{\partial s}(x,s,Du(x_1)) \geq \sigma .}
	\\If $b$ is  jointly concave (i.e. $\mathcal{JC}_b \geq 0$), then
	\bgs{ \mathcal C_u (x_1,x_3,\lambda) \leq  \frac{1}{\sigma}\left[C \varepsilon(Du(x_1)) + \frac{C^2\varepsilon^2(Du(x_1))}{ \nu}\right],
	} 
	otherwise
	\bgs{ \mathcal C_u (x_1,x_3,\lambda) \leq  \frac{1}{\sigma}\left[ -\mathcal{HC}_{b(\cdot, u(\cdot),\xi)} (x_1,x_3,\lambda) + C\varepsilon(Du(x_1)) \left( 1- \frac{\mathcal {JC}_{b(\cdot, u(\cdot), \xi)}(x_1,x_3,\lambda)}{\nu}\right)+ \frac{C^2\varepsilon^2(Du(x_1))}{ \nu}\right],
	} 
	where notations \eqref{epsi} and \eqref{const} are in place.
\end{theorem}
\begin{proof}
	
	As in Theorem \ref{lemma1}, we denote by $\xi$ the common value of $Du$ at the points $x_1,x_2,x_3$.
	Let us also define the $ 2n\times 2n $ matrices
	\begin{align*}
		C:=[D^{2}\mathcal C_u(x_{1},x_{3},\lambda)]=
		\left[\begin{matrix}
			D^{2}_{x_{1}}\mathcal C_u(x_{1},x_{3},\lambda) & D^{2}_{x_{1},x_{3}}\mathcal C_u(x_{1},x_{3},\lambda) \\ D^{2}_{x_{1},x_{3}}\mathcal C_u(x_{1},x_{3},\lambda) & D^{2}_{x_{3}}\mathcal C_u(x_{1},x_{3},\lambda)
		\end{matrix}\right] 
	\end{align*}
	(which is negative semidefinite since $ (x_{1},x_{3},\lambda) $ is a maximum for $ \mathcal C_u $ in the interior), and 
	\begin{align*}
		B:=\left[\begin{matrix}
			s^{2}a^{ij}(x_{2},\xi) & sta^{ij}(x_{2},\xi) \\ sta^{ij}(x_{2},\xi) & t^{2}a^{ij}(x_{2},\xi)
		\end{matrix}\right] 
	\end{align*}
	for $ s,t\in\mathbb{R} $. The matrix $ B $ is positive semidefinite by hypothesis,  therefore 
	the trace of $BC$ is non-negative. 
	That is, denoting
	$$ \alpha:={\rm Tr}(a^{ij}(x_{2},\xi)D^{2}_{x_{1}}\mathcal C_u), \quad
	\beta:={\rm Tr}(a^{ij}(x_{2},\xi)D^{2}_{x_{1},x_{3}}\mathcal C_u), $$
	$$
	\gamma:={\rm Tr}(a^{ij}(x_{2},\xi)D^{2}_{x_{3}}\mathcal C_u), $$
	we have that
	\[
	\alpha s^{2}+2\beta st+\gamma t^{2}\leq0,
	\]
	i.e.
	\begin{equation} \label{negsemdfn}
		\alpha, \gamma \leq 0, \qquad \beta^2 \leq \alpha \gamma.
	\end{equation} 
	Then we obtain
	\begin{equation*}
		\begin{aligned}
			&\alpha=(1-\lambda)^{2}a^{ij}(x_{2},\xi)u_{ij}(x_{2})-(1-\lambda)a^{ij}(x_{2},\xi)u_{ij}(x_{1}),
			\\
			&\gamma=\lambda^{2}a^{ij}(x_{2},\xi)u_{ij}(x_{2})-\lambda a^{ij}(x_{2},\xi)u_{ij}(x_{3}),
			\\
			&\beta=\lambda(1-\lambda)a^{ij}(x_{2},\xi)u_{ij}(x_{2}).
		\end{aligned}
	\end{equation*}
	Denote for $k\in \{1,2,3\}$
	\begin{equation*}
		Q_k=a^{ij}(x_k,Du(x_k)) u_{ij}(x_k)
	\end{equation*}
	and use once more the notations in \eqref{err13}.
Then we have that
	\begin{equation*}
		\begin{aligned}
			&\alpha=(1-\lambda)^{2}Q_2-(1-\lambda)(Q_1+e_1),
			\\
			&\gamma=\lambda^{2}Q_2-\lambda (Q_3+e_3),
			\\
			&\beta=\lambda(1-\lambda)Q_2.
		\end{aligned}
	\end{equation*}
	Using \eqref{negsemdfn},
	we obtain
	\begin{equation}\label{postuff}
		\begin{aligned}
			Q_2\leq \frac{1}{1-\lambda}(Q_1+e_{1}), \qquad
			Q_2\leq\frac{1}{\lambda}(Q_3+e_{3}),
		\end{aligned}
	\end{equation}
	and
	\[ Q_2 \Big((1-\lambda) Q_3 +\lambda Q_1\Big) \leq Q_1 Q_3 + e_3\Big( -(1-\lambda) Q_2 +Q_1+e_1\Big) + e_1\Big( -\lambda Q_2 +Q_3+e_3\Big) - e_1e_3.\]  
	Recalling that $b>0$, hence $(1-\lambda)Q_3+\lambda Q_1> 0$, then 
	\[  Q_2\leq\frac{Q_1Q_3}{(1-\lambda)Q_3+\lambda Q_1}+\frac{e_3\Big( -(1-\lambda) Q_2 +Q_1+e_1\Big) + e_1\Big( -\lambda Q_2 +Q_3+e_3\Big) - e_1e_3}{(1-\lambda)Q_3+\lambda Q_1}. \]
	Denoting
	\[ \zeta(x_1,x_3,\lambda):=\frac{e_3\Big( -(1-\lambda) Q_2 +Q_1+e_1\Big) + e_1\Big( -\lambda Q_2 +Q_3+e_3\Big) - e_1e_3}{(1-\lambda)Q_3+\lambda Q_1},\] we use the equation \eqref{eqn}  and get that
	\bgs{ 
		&\; b(x_2, u(x_2), \xi) - b(x_2,(1-\lambda) u(x_1) + \lambda u(x_3),\xi ) \\
		\leq &\; \frac{b(x_{1},u(x_{1}),\xi)b(x_{3},u(x_{3}),\xi)}{(1-\lambda)b(x_{3},u(x_{3}),\xi)+\lambda b(x_{1},u(x_{1}),\xi)}-  b(x_2,(1-\lambda) u(x_1) + \lambda u(x_3),\lambda ) 
		+\zeta(x_1,x_3,\lambda). } 
	According to \eqref{varfi}, \eqref{harmonic} and using the Lagrange theorem, we have that
	\eqlab{\label{uuu2} & \partial_s b(x_2, \bar s, \xi) \mathcal C_u(x_1,x_3,\lambda) \leq  - \mathcal{HC}_{b(\cdot, u(\cdot),\xi)} (x_1,x_3,\lambda) 
		+\zeta(x_1,x_3,\lambda), } 
	for some $\bar s$ on the segment $[u(x_{2}),\lambda u(x_{3})+(1-\lambda)u(x_{1})] $.
	To estimate $\zeta(x_1,x_3,\lambda)$, we use \eqref{stima99} together with	
	\eqref{postuff} and get that
	\eqlab{\label{uuu1} \zeta(x_1,x_3,\lambda) \leq &\; \frac{|e_3|\Big( -(1-\lambda) Q_2 +Q_1+e_1\Big) + |e_1|\Big( -\lambda Q_2 +Q_3+e_3\Big) - e_1e_3}{(1-\lambda)Q_3+\lambda Q_1}
		\\
		= &\; C\frac{\varepsilon(\xi) \Big( \lambda Q_1 + (1-\lambda)Q_3 + (1-\lambda) Q_1+ \lambda Q_3 - Q_2\Big) + (|e_3|e_1+|e_1|e_3-e_1e_3 ) }{(1-\lambda)Q_3+\lambda Q_1}
		\\
		\leq &\;C\varepsilon(\xi)\left(1 + \frac{ (1-\lambda) Q_1+ \lambda Q_3 - Q_2}{(1-\lambda)Q_3+\lambda Q_1}\right) + \frac{C^2\varepsilon^2(\xi)}{\nu}
		,} 
	using also \eqref{onb} and that
	\[ |e_3|e_1+|e_1|e_3|- e_1e_3\leq  |e_1e_3| \leq C^2\varepsilon^2(\xi).\]
	Now, using again the equation satisfied by $u$,
	notice that
	\bgs{ & (1-\lambda) Q_1+ \lambda Q_3 - Q_2 \\= &\; 
		(1-\lambda) b(x_1, u(x_1), \xi)+ \lambda b(x_3, u(x_3), \xi)- b(x_2, (1-\lambda)u(x_1) + \lambda u(x_3), \xi) \\
		&\; +b(x_2, (1-\lambda)u(x_1) + \lambda u(x_3), \xi) - b(x_2,u(x_2),\xi)
		\\
		= &\;- \mathcal{JC}_{b(\cdot, u(\cdot),\xi)} (x_1,x_3,\lambda)-\partial_s b(x_2,  s, \xi)\mathcal C_u (x_1,x_3,\lambda),
	} 
	according to \eqref{joint} and to Lagrange's theorem.
	Since $\mathcal C_u (x_1,x_3,\lambda) \geq 0$, thanks to \eqref{onderivb} the second term is non-positive, so
	\bgs{ (1-\lambda) Q_1+ \lambda Q_3 - Q_2  \leq - \mathcal{JC}_{b(\cdot, u(\cdot),\xi)} (x_1,x_3,\lambda).}
	Therefore, plugging this into \eqref{uuu1} and \eqref{uuu2}, we have reached
	\bgs{& \partial_s b(x_2, \bar s, \xi) \mathcal C_u(x_1,x_3,\lambda) \leq  - \mathcal{HC}_{b(\cdot, u(\cdot),\xi)} (x_1,x_3,\lambda) + C\varepsilon(\xi) \left(1- \frac{ \mathcal{JC}_{b(\cdot, u(\cdot),\xi)} (x_1,x_3,\lambda)}{(1-\lambda)Q_3+\lambda Q_1}\right) + \frac{C^2\varepsilon^2(\xi)}{\nu}.
	}  
	We point out that
	$\mathcal{HC}_b \geq \mathcal {JC}_b$, hence if $\mathcal {JC}_b\geq 0$, i.e. $b$ is jointly concave, then $b$ is also harmonic concave and in that case,
	\bgs{ \partial_s b(x_2, s, \xi) \mathcal C_u(x_1,x_3,\lambda) \leq  \varepsilon(\xi)  + \frac{\varepsilon^2(\xi)}{\nu}.
	}  
	Otherwise,  if $\mathcal {JC}_b\leq 0$, using also \eqref{onb}, we get that
	\bgs{\mathcal C_u(x_1,x_3,\lambda) \leq \frac{1}{\sigma}\left[
		- \mathcal{HC}_{b(\cdot, u(\cdot),\xi)} (x_1,x_3,\lambda) + C\varepsilon(\xi) \left(1- \frac{ \mathcal{JC}_{b(\cdot, u(\cdot),\xi)} (x_1,x_3,\lambda)}{\nu} \right)+ \frac{C^2\varepsilon^2(\xi)}{\nu} \right].
	}  
		This concludes the proof. \qedhere

	\end{proof}

	\section{Application to semi-linear equations}\label{sec2}

 We investigate two applications of our general maximum convexity principles. We point out that the characteristics of these applications, particularly the boundary conditions, drive the convexity function $C_u$ of the solution to attain a positive maximum within the interior  of the domain. Then we readily apply Theorems \ref{lemma1}, \ref{lemma2} to obtain the estimates on the loss of concavity of a transformed of the solution $u$. 
 
 It is worth mentioning that, in the classical case, the concavity of the solution depends on the  (harmonic)-concavity of the nonlinearity. In both our subsequent applications, Problem \ref{appl1} and \ref{appl2}, already a direct use of the maximum convexity principles in Theorems \ref{lemma1}, \ref{lemma2} provides this connection. We give a bound on the convexity of the nonlinearity in terms of the spatial variation, to emphasize the role of the introduced anisotropy, see also subsequent Remark \ref{noconc}.

\medskip

In this section, let $\Omega \subset \R^n$ be a bounded open strongly convex set with $C^1$ boundary.

\begin{problem}\label{appl1}

Let 
\[ a\colon \overline \Omega \to (0,+\infty )\] and for all $i,j\in \{1, \dots,n\}$ let the functions
\[ \alpha^{ij} \colon \overline \Omega  \to (0,+\infty)\] be such that
there exists $\zeta>0$ such that 
\bgs{
	\sum_{i,j=1}^n\alpha^{ij}(x)p_ip_j \geq\zeta |p|^2,\qquad \mbox{ for all } p \in\R^n,
	}
	and $a,\alpha^{ij} (\cdot) \in C^1( \Omega)$.
	 Consider the equation 
	\sys[]{\label{pop1}
	& - \sum_{i,j=1}^n \alpha^{ij}(x)D^2_{ij}u=a(x) \, u^\beta,\qquad \beta\in [0,1) && \mbox{ in } \Omega\\
	& u>0 &&\mbox{ in } \Omega\\
	& u=0 && \mbox{ on } \partial \Omega.
	}
\end{problem}	
 
 \noindent
 We recall the following property of convex sets \cite{Kor83Co}.

 	\begin{proposition}\label{domm}
 		Let $\Omega \subset \R^n$ be bounded strongly convex set with $C^1$ boundary. Then there exist $r_0>0$ such, that for every $\rho\in (0,r_o]$, the set
 		$$
 		\Omega_\rho:=\big\{x\in\Omega:  d(x,\partial\Omega)> \rho\big\}.
 		$$
 		is convex with $C^1$ boundary.
 	\end{proposition}
 
%
%
%
%

\vskip4pt

We recall once more that, when the coefficients $\alpha^{ij}$ do not depend on $x$ and when $a(x)=1$, the power function $u^\alpha$, for some $\alpha:=\alpha(\beta)$, is concave. We want to understand the impact of introducing a dependency on $x$ in the equation. We are able to obtain a precise quantitative result about the extent to which the transformed solution deviates from concavity.

\begin{proposition}
	\label{mainappl}
	Let $u\in C^2(\Omega) \cap C^1(\overline \Omega)$ be a solution of Problem \ref{appl1}, assuming additionally 
	that 
$$
\|\nabla a\|_{L^\infty(\Omega)}+\max_{i,j\in \{ 1,\dots, n\} }\|\nabla \alpha^{ij}\|_{L^\infty(\Omega)}< \epsilon,
$$
for some $\epsilon>0$.
Then
\[ \max_{ \overline\Omega \times\overline\Omega \times [0,1] } \mathcal C_{-u^{\frac{1-\beta}2} } (x,y, t) \leq  C \epsilon,
\]
for some $ C>0$. 
%
\end{proposition}

	\begin{proof}Let
	$$
	v:=-u^{\frac{1-\beta}{2}}.
	$$ 
	We point out that $v<0$, $v\in C^2(\Omega)$ 
  and we focus on deriving  the equation satisfied by $v$.
	We have that $ u=(-v)^{\frac{2}{1-\beta}} $, and 
	\begin{equation*}
		D_{i}u=-\frac{2}{1-\beta}(-v)^{\frac{1+\beta}{1-\beta}}D_{i}v 
	\end{equation*} 
	where $ D_{i}=\frac{\partial}{\partial x_{i}} $, and
	\begin{equation*}
		D_{ij}^{2}u=\frac{2(1+\beta)}{(1-\beta)^{2}}(-v)^{\frac{2\beta}{1-\beta}}D_{i}vD_{j}v-\frac{2}{1-\beta}(-v)^{\frac{1+\beta}{1-\beta}}D_{ij}^{2}v
	\end{equation*}
	where $ D_{ij}^{2}=\frac{\partial^{2}}{\partial x_{i}\partial x_{j}} $. 
	This gives that
	$$\sum_{i,j=1}^n \alpha^{ij}(x)D_{ij}^{2}u=\frac{2(1+\beta)}{(1-\beta)^{2}}(-v)^{\frac{2\beta}{1-\beta}}\sum_{i,j=1}^n\alpha^{ij}(x)D_ivD_jv-\frac{2}{1-\beta}(-v)^{\frac{1+\beta}{1-\beta}}\sum_{i,j=1}^n\alpha^{ij}(x)D^2_{ij} v. 
	$$ 
	Thus we obtain
	\begin{align*}
		-\frac{2(1+\beta)}{(1-\beta)^{2}}(-v)^{\frac{2\beta}{1-\beta}}\sum_{i,j=1}^n\alpha^{ij}(x)D_ivD_jv+\frac{2}{1-\beta}(-v)^{\frac{1+\beta}{1-\beta}}\sum_{i,j=1}^n\alpha^{ij}(x)D^2_{ij} v=a(x)(-v)^{\frac{2\beta}{1-\beta}}.
	\end{align*}
	Dividing by $\frac{2}{1-\beta}(-v)^{\frac{1+\beta}{1-\beta}}$ yields
	\eqlab{\label{oop1}
		\sum_{i,j=1}^n\alpha^{ij}(x)D^2_{ij} v=(-v)^{-1}\Big(\frac{a(x)(1-\beta)}{2}+\frac{1+\beta}{1-\beta}\sum_{i,j=1}^n\alpha^{ij}(x)D_ivD_jv\Big),
	}
	that is 
	\begin{align*}
		\sum_{i,j=1}^n\alpha^{ij}(x)D^2_{ij} v-b(x,v,Dv)=0,
	\end{align*}
	where 
	\eqlab{\label{bprimo22}
	b(x,s,\xi):=(-s)^{-1}f_\xi(x) 
	,}
	and for any $\xi \in \R^n$, $f_\xi\colon \Omega  \to \R$ is 
\[ f_\xi(x):= \frac{a(x)(1-\beta)}{2}+\frac{1+\beta}{1-\beta}
	\sum_{i,j=1}^n\alpha^{ij}(x)\xi_i\xi_j.\]
If $\mathcal C_{-u^{(1-\beta)/2}} \leq 0$ in $\overline \Omega\times\overline \Omega \times [0,1]$, then there is nothing to prove. Otherwise,  from \cite[Corollary 3.2]{BS20}, we have that $\mathcal C_{-u^{(1-\beta)/2}}$, cannot achieve any positive maximum on the boundary, i.e. the positive maximum of $\mathcal C_{v}$ is attained at some point $(x_1,x_3,\lambda)\in \Omega \times \Omega \times (0,1)$.  
Recalling Proposition \ref{domm}, let
\[ \rho:= \min\{ r_0, d(x_1, \Omega), d(x_3,\Omega)\},\] then
$(x_1,x_3,\lambda)\in  \overline \Omega_\rho  \times \overline \Omega_\rho \times (0,1)$
and define
	\[
	  m_\rho:=\|v\|_{C(\overline\Omega_\rho)}, \, M_\rho:=\|D v\|_{C(\overline\Omega_\rho)}. \]
	  Notice that
	\[f_\xi(x) \geq  \frac{1+\beta}{1-\beta}\zeta|\xi|^2+\frac{1-\beta}{2}\min_{\overline \Omega_\rho} a(x)\geq \frac{1-\beta}{2}\min_{\overline \Omega_\rho} a(x)>0. \]	
We have that for all $x\in \overline{\Omega}_\rho,  s\in [-m_\rho,0) ,\xi \in  \overline B_{M_\rho}$ 
\[ b(x,s,\xi)  \geq \frac{1-\beta}{2m_\rho} \min_{\overline \Omega_\rho} a(x) :=\nu>0, \qquad  \partial_s b (x,s, \xi) \geq \frac{1-\beta}{2m_\rho^2} \min_{\overline \Omega_\rho} a(x) :=\sigma >0.\] 
For clarity, we point out that we have $[x_1,x_3]\subset \overline \Omega_\rho$, $[v(x_1),v(x_3)] \subset [-m_\rho,0)$  and $Dv(x_1)\in \overline B_{M_\rho}$, thus the hypothesis \eqref{onb}, \eqref{onderivb} in Theorem \ref{lemma2} are fulfilled.
Denote for all $\xi \in \overline B_{M_\rho}$,
\[ \mathfrak{m} :=\min_{\overline \Omega_\rho}  f_\xi(x), 
 \, \, \mathfrak{M} = \max_{\overline \Omega_\rho} f_\xi(x)
 \] 
 and remark that, for some $\bar x, \tilde x \in \overline\Omega_\rho$ and $\bar z$ on the segment $[\bar x, \tilde x]$ lying in $\overline \Omega_\rho$,
 \bgs{  \mathfrak{M}-\mathfrak{m} = &\; f_\xi(\bar x) - f_\xi(\tilde x) \leq |\nabla f_\xi(\bar z)| \mbox{diam}(\Omega_\rho) 
 \\ 
 \leq &\; \epsilon\left( \frac{1-\beta}{2} + \frac{1+\beta}{1-\beta}n^2M_\rho^2\right) := \epsilon \mathfrak C
 .} 
Let $\upsilon_\rho:= \min_{\overline\Omega_\rho} (-v)>0$, there holds\footnote{We thank Marco Gallo for pointing out a preliminary version of this estimate.}
\bgs{
& \mathcal{HC}_{b(\cdot, v(\cdot), \xi)}(x_1,x_3,\lambda)	
\\ \geq &\; \frac{1}{\lambda (-v)(x_3)+ (1-\lambda) (-v)(x_1)}\\\; \; & \; \left(  f_\xi\big((1-\lambda) x_1 + \lambda x_3\big)  - f_\xi(x_1) f_\xi(x_3) \frac{\lambda (-v)(x_3)+ (1-\lambda) (-v)(x_1)}{\lambda f_\xi(x_1)(-v)(x_3) +(1-\lambda) f_\xi(x_3)(-v)(x_1)} \right)
\\
\geq &\; \frac{1}{\lambda (-v)(x_3)+ (1-\lambda) (-v)(x_1)} \left(  \mathfrak{m}- \frac{f_\xi(x_1)f_\xi(x_3)}{\mathfrak{m}} \right)
\\
\geq&\;  \frac{1}{\lambda (-v)(x_3)+ (1-\lambda) (-v)(x_1)} \left(  \mathfrak{m}- \frac{\mathfrak{M}^2}{\mathfrak{m}}\right)
\\
\geq&\; \frac{1}{\upsilon_\rho}\frac{ \mathfrak{m}^2-\mathfrak{M}^2}{\mathfrak{m}} = -\epsilon\frac{\mathfrak C}{\upsilon_\rho}\frac{\mathfrak{M}+\mathfrak{m}}{\mathfrak{m}}.}
Also we have that
\[ \mathcal{JC}_{b(\cdot, (-v)(\cdot), \xi)}(x_1,x_3,\lambda)	\geq \frac{\mathfrak{m}}{m_\rho} - \frac{\mathfrak{M}}{\upsilon_\rho} :=-\alpha.\]
 According to Theorem \ref{lemma2}, we have that for all $(x,y, t) \in \overline \Omega \times \overline\Omega \times [0,1]$,
 \bgs{ & \max_{ \overline \Omega \times \overline\Omega \times [0,1]} \mathcal C_{-u^{\frac{1-\beta}{2}}}
  (x,y,t)  = \mathcal C_{-u^{\frac{1-\beta}{2}}}
  (x_1,x_3,\lambda)  \\
  = &\; \mathcal C_{v} 
  (x_1,x_3,\lambda) \leq  \frac{1}{\sigma}\left[ \frac{C}{\upsilon_\rho} \frac{\mathfrak{M}+\mathfrak{m}}{\mathfrak{m}}\epsilon + C\epsilon \left( 1- \frac{\alpha}{\nu}\right)+ \frac{C^2\epsilon}{ \nu}\right]:=  C_\rho\epsilon
 ,
	} 
	when $\epsilon$ is small enough. 
\end{proof}



\begin{remark}\label{noconc}
We point out 
the difference with what is obtained for the autonomous model case $-\Delta u=u^\beta$. 
There, the transformation $u^{(1-\beta)/2}$ is concave, since the right hand side of \eqref{oop1}, 
the transformed equation, is  harmonic concave. In our case, we control the "loss of concavity" by the variation of the introduced anisotropy, and this is the best one can hope for: the function $b$ defined in \eqref{bprimo22} is never harmonic concave, jointly in the two variables $(x,s)$. Indeed, it is known that a positive function $b$ is harmonic concave if and only if $B=1/b$ is convex. However, even in the plane, the hessian of the function  $B(x,s) =s g(x) $ is negatively defined, hence $B$ is nor convex, nor concave, unless $g$ is constant.
\end{remark}

\begin{problem}\label{appl2}
Let 
\[ a\colon \overline \Omega \to (0,+\infty )\] and for all $i,j\in \{1, \dots,n\}$ let the functions
\[ \alpha^{ij} \colon \overline \Omega  \to (0,+\infty)\] be such that
there exists $\zeta>0$ such that 
\bgs{
	\sum_{i,j=1}^n\alpha^{ij}(x)p_ip_j \geq\zeta |p|^2,\qquad \mbox{ for all } p \in\R^n,
	}
	and $a,\alpha^{ij} (\cdot) \in C^1( \Omega)$.
	Let $\varphi\colon (0,+\infty)\to (0,+\infty)$ be such that $\varphi \in C^1( 0,+\infty)$ and $\varphi'(t)\leq 0$.
	Consider, for $\epsilon>0$, the problem
	\sys[]{\label{pop1}
	& - \sum_{i,j=1}^n \alpha^{ij}(x)D^2_{ij}u=a(x)u+\epsilon \varphi(u) && \mbox{ in } \Omega\\
	& u>0 &&\mbox{ in } \Omega\\
	& u=0 && \mbox{ on } \partial \Omega.
}
This can be considered as a perturbation of the nonautonomous version of the eigenvalue problem
for second order elliptic operators. We remark that the condition on $\varphi$ can be loosened to accommodate other perturbations, in particular one can require  $\varphi\in C^1(\R^+,\R^+)$ be such that, 
		$ 
	e^s \varphi(e^{-s})-\varphi'(e^{-s})\geq \gamma>0,$ for all  $ s\geq -R$, for some $R>0$, e.g. $\varphi(t)= t^\beta, \beta \in [0,1)$ can also be considered.  
\end{problem}	

We have the following result. 
\begin{proposition}\label{ghfd}
	Let $u\in C^2(\Omega)\cap C(\overline \Omega)$ be a solution of 
	Problem \ref{appl2}. Assume that 
		$$
\|\nabla a\|_{L^\infty(K\Omega)}+\max_{i,j\in \{ 1,\dots, n\} }\|\nabla \alpha^{ij}\|_{L^\infty(\Omega)}< \epsilon.
$$
Then
\[ \max_{ \overline\Omega \times\overline\Omega \times [0,1] } \mathcal C_{-\log u } (x,y, t) \leq  C \epsilon,
\]
for some $ C>0$. 
\end{proposition} 	
\begin{proof}
	Letting $u= e^{-v}$ we have $v= -\log u$.
	We notice that $v\in C^2(\Omega)$ and that
	by a direct calculation we obtain
\[
\sum_{i,j=1}^n \alpha^{ij}(x)D^2_{ij}v= b(x,v,Dv),
\]
where
$$
b(x,s,\xi):=\sum_{i,j=1}^n \alpha^{ij}(x)\xi_i  \xi_j +a(x)+\epsilon  e^{s}\varphi(e^{-s}).
$$
If $\mathcal C_v\leq 0$ in $\overline \Omega\times\overline \Omega \times [0,1]$, then there is nothing to prove. Otherwise, using \cite[Lemma 3.11]{Kaw85R}, we have that 
$\mathcal C_{v}$, cannot achieve any positive maximum on the boundary, i.e. the maximum of $\mathcal C_{v}$ is attained at some point $(x_1,x_3,\lambda)\in \Omega \times \Omega \times (0,1)$. 
Recalling Proposition \ref{domm}, let
\[ \rho:= \min\{ r_0, d(x_1, \Omega), d(x_3,\Omega)\},\] then
$(x_1,x_3,\lambda)\in  \overline \Omega_\rho  \times \overline \Omega_\rho \times (0,1)$
and define
\[
m_\rho:=\|v\|_{C(\overline\Omega_\rho)}, \, M_\rho:=\|D v\|_{C(\overline\Omega_\rho)}. 
\]
Notice that
for all $x\in \overline \Omega_\rho,  s\in [ -m_\rho, m_\rho],\xi \in \overline B_{M_\rho}$, using the hypothesis on $\varphi$,
\[
\partial_s b(x,s,\xi)\geq 	\epsilon( e^s \varphi(e^{-s})-\varphi'(e^{-s}))\geq \epsilon e^{-m_\rho} \varphi(e^{m_\rho}):=\sigma.
\] 
For clarity, we observe that we have $[x_1,x_3]\subset \overline \Omega_\rho$, $[v(x_1),v(x_3)] \subset [-m_\rho,m_\rho]$  and $Dv(x_1)\in \overline B_{M_\rho}$, thus the hypothesis \eqref{cres} in Theorem \ref{lemma1} are fulfilled.
We have that
\bgs{
	|\mathcal {JC}_{b(\cdot, v(\cdot),\xi)}(x_1,x_3,\lambda)| 
	\leq &\;  C(\|\nabla \alpha^{ij} \|_{L^\infty(\Omega_\rho)} + \|\nabla a\|_{L^\infty(\Omega_\rho)} + \epsilon) \leq C \epsilon.}
Applying Theorem \ref{lemma1} yields the conclusion.
\end{proof}	

\vskip4pt

Finally, we recall the following \cite[Theorem 2]{ulam}
\begin{proposition}(Hyers-Ulam)  
	\label{ulam}
	Let $ X $ be a space of finite dimension and $ D \subset X $ convex. Assume that $ f : D \rightarrow \mathbb{R} $ is $ \delta  $-convex, i.e. for all $(x,y,t)\in D\times D \times [0,1]$
	\[ \mathcal C_f (x,y,t)\leq  \delta.\] Then there exists 
	a convex function $g : D \rightarrow \mathbb{R} $ such that $ \|f-g\|_{L^{\infty}(D)}\leq \delta k_{n},$ where $  k_{n}>0 $ depends only on $ n=dim(X) $.
\end{proposition}	

By using this result, based upon the estimates of Propositions \ref{mainappl} and \ref{ghfd} we obtain the approximate concavity results stated in the introduction
for the transformations  $u^{(1-\beta)/2}$ in the case $\beta\in (0,1)$ and $\log u$ for the case $\beta=1$.

\begin{remark}
	The constant $C$ appearing in the conclusions of Proposition \ref{mainappl} and \ref{ghfd} is related to the $C$ 
	introduced in formula \eqref{const} which depends on $n$, ${\rm diam}(\Omega_\rho)$ and on the supremum norms of the second order derivatives
	of the transformation $v$ on $\Omega_\rho$ and hence (since $u$ is bounded away from $0$ on $\Omega_\rho$) 
	on the supremum norms of $D_{ij}^2u$ on $\Omega_\rho$. 
	By the classical Schauder estimates for second order linear elliptic
	operators (see \cite[Theorem 6.2]{gilbarg}), in turn $C$ depends on $n$, ${\rm diam}(\Omega_\rho)$, the ellipticity constant $\zeta$, $\|\alpha^{ij}\|_{C^{0,\alpha}(\Omega_\rho)},$ 
	$\|a\|_{C^{0,\alpha}(\Omega_\rho)}$ and  $\|u\|_{C^{0,\alpha}(\Omega_\rho)}$ for some $\alpha\in (0,1)$.
\end{remark}


\bigskip
\bigskip

\begin{thebibliography}{10}
	
%
%
%
%
%
%
%
%
%
%
%
%
%
%
%
%
%
%
%
%
%
%
%
%
%
%
%
%
%
%

%
%



\bibitem{AAGS}
N. Almousa, J. Assettini, M. Gallo, M. Squassina,
\emph{Concavity properties for quasilinear equations and optimality remarks},
Differential Integral Equations (2023), to appear.




\bibitem{BMS22} W.\ Borrelli, S.\ Mosconi, M.\ Squassina,
\emph{Concavity properties for solutions to $p$-Laplace equations with concave nonlinearities},
Adv. Calc. Var. (2022), \href{https://doi.org/10.1515/acv-2021-0100}{doi.org/10.1515/acv-2021-0100}.

\bibitem{BL76} H.\ J.\ Brascamp, E.\ H.\ Lieb,
\emph{On extensions of the Brunn-Minkowski and Prékopa-Leindler theorems, including inequalities for log concave functions, and with an application to the diffusion equation},
J. Funct. Anal. \textbf{22} (1976), 366--389.

%


\bibitem{BS20} C.\ Bucur, M.\ Squassina,
\emph{Approximate convexity principles and applications to PDEs in convex domains},
Nonlinear Analysis \textbf{192} (2020), article ID 111661. %





%




\bibitem{GNN79} B.\ Gidas, W.-M.\ Ni, L.\ Nirenberg, 
\emph{Symmetry and related properties via the maximum principle},
Commun. Math. Phys. \textbf{68} (1979), 209--243.


\bibitem{gilbarg}
D. Gilbarg, N. Trudinger,
Elliptic partial differential equations of second order, Springer Verlag, 1977.

\bibitem{Gom07} J.\ M.\ Gomes,
\emph{Sufficient conditions for the convexity of the level sets of ground-state solutions},
Arch. Math. \textbf{88} (2007), 269--278.

%





\bibitem{HNST18} A.\ Henrot, C.\ Nitsch, P.\ Salani, C.\ Trombetti,
\emph{Optimal concavity of the torsion function},
J. Optim. Theory Appl. \textbf{178} (2018), 26--35.

\bibitem{ulam} D.H. Hyers, S.M. Ulam,
\emph{Approximately convex functions},
Proc. Amer. Math. Soc., \textbf{3} (1952), 821–-828.






\bibitem{Kaw85R} B.\ Kawohl, 
``Rearrangements and convexity of level sets in PDE'',
Lecture Notes in Math. \textbf{1150}, Springer-Verlag Heidelberg, 1985.



\bibitem{Kaw85W} B.\ Kawohl, 
\emph{When are solutions to nonlinear elliptic boundary value problems convex?},
Comm. Partial Differential Equations \textbf{10} (1985), 1213--1225.



\bibitem{Ken85} A.\ U.\ Kennington,
\emph{Power concavity and boundary value problems},
Indiana Univ. Math. J. \textbf{34} (1985), 687--704.


\bibitem{Kor83Co} N.\ J.\ Korevaar, 
\emph{Convex solutions to nonlinear elliptic and parabolic boundary value problems},
Indiana Univ. Math. J. \textbf{32} (1983), no. 4, 603--614.


\bibitem{Lew77} J.\ L.\ Lewis, 
\emph{Capacitary functions in convex rings},
Arch. Ration. Mech. Anal. \textbf{66} (1977), 201--224.

\bibitem{Lin94} C.-S.\ Lin,
\emph{Uniqueness of least energy solutions to a semilinear elliptic equation in $\R^2$},
Manuscripta Math. \textbf{84} (1994), 13--19.

\bibitem{Lind94} P.\ Lindqvist,
\emph{A note on the nonlinear Rayleigh quotient},
Potential Anal. \textbf{2} (1993), 199--218.

\bibitem{Lio81} P.-L. Lions, 
\emph{Two geometrical properties of solutions of semilinear problems},
Appl. Anal. \textbf{12} (1981), 267--272.
%



\bibitem{ML71} L.\ G.\ Makar-Limanov,
\emph{Solution of Dirichlet’s problem for the equation $\Delta u = -1$ in a convex region},
Mat. Zametki \textbf{9} (1971), 89--92.

\bibitem{PS51} G.\ Pólya, G.\ Szegö, 
``Isoperimetric inequalities in mathematical physics'',
Ann. of Math. Stud. \textbf{27}, Princeton University Press, 1951. 

\bibitem{Sak87} S.\ Sakaguchi,
\emph{Concavity properties of solutions to some degenerate quasilinear elliptic Dirichlet problems},
Ann. Sc. Norm. Super. Pisa Cl. Sci. 4 \textbf{14} (1987), 403--421.
	
\end{thebibliography}
\end{document}